\documentclass[12pt, reqno]{amsart}
\usepackage{amsmath}
\usepackage{amssymb}
\usepackage{amsthm}
\usepackage{enumerate}
\usepackage[mathscr]{eucal}
\usepackage{xcolor}
\theoremstyle{plain}
\usepackage{tikz}
\usepackage{soul}
\usepackage[normalem]{ulem}
\newtheorem{theorem}{Theorem}[section]
\newtheorem{lemma}[theorem]{Lemma}
\newtheorem{prop}[theorem]{Proposition}
\theoremstyle{definition}
\newtheorem{definition}[theorem]{Definition}
\newtheorem{remark}[theorem]{Remark}

\newtheorem{cor}[theorem]{Corollary}
\theoremstyle{remark}

\begin{document}
\title[On symmetricity of orthogonality]{On symmetricity of orthogonality with respect to numerical radius norm}
	
	\author[ Sohel ]{Shamim Sohel  }
	
	\thanks{The author would like to thank ANRF, Govt. of India, for the financial support in the form of National Post Doctoral Fellowship (File No.: PDF/2025/001669) under the mentorship of Dr. Surjit Kumar. }
	
	\address[Sohel]{Department of Mathematics, Indian Institute of Technology Madras, Chennai- 600036,  India.}	\email{shamimsohel11@gmail.com}

		\subjclass[2020]{Primary 47L05, Secondary 46B20}
	\keywords{Birkhoff-James orthogonality,   numerial radius norm, left symmetric points, right symmetric points, space of all continuous functions}
	
	\maketitle
	
	\begin{abstract}
	We study Birkhoff-James orthogonality and local symmetric points in the space $C(K,X)$ of vector-valued continuous functions equipped with a 	numerical-radius type (semi-)norm. We extend the theory of abstract numerical ranges to semi-normed spaces and obtain complete characterizations of the left and right symmetric points of $C(K,X)$ with respect to the numerical-radius (semi-) norm. We also establish corresponding results for the space  $\mathcal{K}(X)$ of compact operators on $X$.
	\end{abstract}

\section{Introduction}

Let $X$ be a Banach space over the field $\mathbb K$. We denote by
$S_X$ and $B_X$ the unit sphere and the closed unit ball of $X$,
respectively.  An element
$x\in X$ is said to be \emph{Birkhoff-James orthogonal} to
$y\in X$, denoted by $x\perp_B y$, if
\[
\|x+\lambda y\|\geq \|x\|,
\qquad
\forall\,\lambda\in\mathbb K.
\]
Introduced by Birkhoff \cite{B} in the setting of normed linear spaces and later studied  by James \cite{J}, this notion provides a natural extension of the usual orthogonality from Hilbert spaces to arbitrary normed spaces. Indeed, whenever the norm is induced by an inner product, $
x\perp_B y
\Longleftrightarrow
\langle x,y\rangle=0.$
Unlike Hilbert spaces, however Birkhoff-James orthogonality is generally not symmetric in general Banach spaces. This led to the introduction of the notions of local symmetric points in \cite{S1}.

\begin{definition}
	An element $x\in X$ is called a \emph{left symmetric point} if
	$x\perp_B y$ implies $y\perp_B x$ for every $y\in X$.
	Similarly, $x$ is called a \emph{right symmetric point} if
	$y\perp_B x$ implies $x\perp_B y$ for every $y\in X$.
	An element is called a \emph{symmetric point} if it is both left and right symmetric.
\end{definition}

The study of symmetric points has attracted considerable attention in the geometry of Banach spaces due to its close connection with the structure of normed spaces and the theory of isometries; see, for example, \cite{BRS, CSS,GSP,PMW,S1,SRBB, SSP, T}. Although abstract characterizations of symmetric points are known (see \cite{MPS}), explicit descriptions have been obtained only in certain classes of Banach spaces, including some sequence spaces (\cite{BRS}) and some special spaces of bounded operators (\cite{PMW, PSS}). Determining the symmetric points explicitly in a given Banach space can therefore become a challenging problem due its involved connection with the geometry of the space.  In this article, we study Birkhoff-James orthogonality and the local symmetric points in the space of vector-valued continuous functions endowed with a numerical-radius-type (semi-)norm. This extends our earlier work on $C(K,X)$  endowed with the usual supremum norm to a substantially different geometric setting (see \cite{PSS}). To motivate this construction, we briefly recall the notions of abstract numerical ranges and the classical numerical radius.

For $u\in S_X$, let $ J(u) = \{x^*\in S_{X^*}:x^*(u)=1\}, $
the set of all support functionals at $u$.
The \emph{abstract numerical range} of $x\in X$ relative to $u$ is
defined by
\[
V(X,u,x)= \{x^*(x):x^*\in J(u)\},
\]
and the associated \emph{abstract numerical radius} is
\[
v_u(x)= \sup_{x^*\in J(u)}|x^*(x)|.
\]
Clearly $v_u(\cdot)$ is a semi-norm on $X$. Moreover,
$v_u(\cdot)$ is a norm if and only if $J(u)$ separates the points of
$X$. Such a point $u \in S_X$ is called a \emph{vertex} of the unit ball.
Observe that every vertex is necessarily an extreme point of $B_X$.
Conversely, if $X$ is finite-dimensional and polyhedral, then every
extreme point of $B_X$ is a vertex. The abstract numerical radius has
been studied extensively and has found numerous applications in Banach
space geometry, see \cite{MMQRS}.

We also recall the classical numerical range on $\mathcal{L}(X)$, the space of all bounded linear operators on $X$.
For $T\in \mathcal{L}(X)$, the numerical range is defined by
\[
W(T) = \{x^*(Tx): x\in S_X,\,x^*\in S_{X^*},\,x^*(x)=1\},
\]
and the corresponding numerical radius is
\[
w(T) =\sup \{|x^*(Tx)|:  x\in S_X,\,x^*\in S_{X^*},\,x^*(x)=1\}.
\]
This classical numerical range is, in fact, a special case of the abstract construction above, for any $T \in \mathcal{L}(X),$    
$$W(T)= V(\mathcal{L}(X), I, T), $$ 
where $I$ is the identity operator on $X.$
The numerical radius $w(\cdot)$ is likewise  a semi-norm on $\mathcal{L}(X)$ and plays a fundamental role in the local theory of operators. Recently, Birkhoff-James orthogonality in $\mathcal{L}(X)$ equipped with the numerical radius (semi-)norm has been studied in \cite{MPS22, RS}.  The corresponding left and right symmetric elements, referred to as \emph{nr-left and nr-right symmetric operators}, respectively, have also been investigated in this setting in \cite{CS, GMPS}.

We now consider a numerical-radius-type semi-norm on the space of vector-valued continuous functions. 
Let $K$ be a compact subset of the unit sphere $S_X$ and consider the
Banach space
\[
C(K,X) = \{f:K\rightarrow X:\,f\text{ is continuous}\}.
\]
For each $k \in K$, let $
J(k) = \{x^*\in S_{X^*}:x^*(k)=1\}, $
and define
\[
\Pi(K,X) = \{(k,x^*):k \in K,\,x^*\in J(k)\}.
\]
For any $f \in C(K, X),$ the numerical range is defined as 
\[
W(f)= \{ x^*(f(k)):(k,x^*)\in\Pi(K,X)\}
\] 
and corresponding numerical radius is 
\[
\|f\|_{w} =\sup \{|x^*(f(k))|:(k,x^*)\in\Pi(K,X)\}, \quad f\in C(K,X).
\]
In general, $\|\cdot\|_{w}$ is only a semi-norm on $C(K,X)$ and may be
viewed as an analogue of the classical numerical radius on $\mathcal{L}(X)$. Numerical ranes and numerical radius of this kind for continuous functions have been considered in several settings, see \cite{B62, BD71, BD73, H, H74, M}. The particular construction considered here was introduced in \cite{MMR}, where it was referred to as the \emph{spatial numerical range}.
Let $f, g \in C(K, X)$, we say that $f$ is Birkhoff-James orthogonal to $g$ with respect to the (semi-)norm $\|\cdot\|_w,$ if 
\[
\|f + \lambda g\|_w \geq \|f\|_w, \quad \, \forall \lambda \in \mathbb{K}.
\] 
We denote it as $ f \perp_w g.$ Similar to BJ-orthogonality in normed spaces, this orthogonality is not symmetric in nature. In analogy with the corresponding notions for $\mathcal{L}(X)$, we say that
$f\in C(K,X)$ is \emph{$nr$-left symmetric} if for any  $g\in C(K,X)$,
\[
f\perp_w g \quad \implies \quad g\perp_w f.
\]
 Similarly, $f$ is said to be
\emph{$nr$-right symmetric} if for any $g\in C(K,X)$,
\[
g\perp_w f \quad \implies \quad f\perp_w g.
\]
We say $f$ is $nr$-symmetric if $f$ is both $nr$-left and $nr$-right symmetric.
The principal aim of this paper is to characterize these symmetric
functions in terms of the geometry of the corresponding pointwise
abstract numerical-radius semi-norms. Our first main result gives a
complete characterization of $nr$-left symmetric functions.

\begin{theorem}\label{main-left}
	Let $X$ be a Banach space and let $K$ be a compact subset of $S_X$. 	Let $f\in C(K,X)$. Then $f$ is $nr$-left symmetric if and only if  there exists $k_0 \in K$ such that the following conditions hold:
	\begin{itemize}
		\item[(i)]  for any $k \neq k_0,$ and for any $x^* \in J(k),$ 	$x^*(f(k))=0$ 
		\item[(ii)] $f(k_0)$ is left symmetric with respect to 	$v_{k_0}(\cdot)$.
	\end{itemize}
\end{theorem}

 The characterization of $nr$-right symmetric functions exhibits a
substantially different behavior.   In contrast to the left symmetric case, the numerical-radius
seminorm of $f$ must be attained at every point of $K$, while each
$f(k)$ must be right symmetric with respect to $v_k(\cdot)$.

\begin{theorem}\label{main-right}
	Let $X$ be a real Banach space and let $K$ be a compact subset of	$S_X$. Let $f\in C(K,X)$. Then $f$ is $nr$-right symmetric if and only 	if  for every $k\in K$ the following conditions hold:
	\begin{itemize}
		\item[(i)] there exists $x^*\in J(k)$ such that	$|x^*(f(k))|=\|f\|_w$
		\item[(ii)] $f(k)$ is right symmetric with 	respect to $v_k(\cdot)$.
	\end{itemize}
\end{theorem}
The proofs of Theorems~\ref{main-left} and~\ref{main-right} are given
in Theorem \ref{left} and \ref{right}, respectively.  In addition, we investigate left and right
symmetric points with respect to the Birkhoff-James orthogonality
induced by the semi-norm $v_u(\cdot)$, which plays a crucial role in
the above characterizations. We subsequently derive several
consequences of these results in some classical Banach
spaces and in Hilbert spaces. Finally, we apply  the results to the space $K(X)$ of compact operators and obtain  corresponding results for compact operators endowed with the
numerical radius semi-norm.

    \section{Preliminaries}
	
	Although Birkhoff-James orthogonality is traditionally defined for
	normed spaces, the same definition applies naturally to semi-normed
	spaces. Let $(X,\eta)$ be a semi-normed space over $\mathbb K$. We say
	that $x$ is Birkhoff-James orthogonal to $y$ with respect to $\eta$,
	denoted by $x\perp_\eta y$, if
	\[
	\eta(x+\lambda y)\geq\eta(x) 	\quad \text{for every } \lambda\in\mathbb K.
	\]
	When $\eta$ is a norm, this is the usual Birkhoff-James orthogonality.
	
	For clarity we will write $(X, \eta)$ as $X_\eta$ and the dual $ (X, \eta)^*$ as $X_\eta^*.$
	Put $N_\eta=\ker\eta$, and let $q:X\to X/N_\eta$ be the quotient map.
	The quotient space $X/N_\eta$, endowed with the norm 
	\[
	\| x + N_\eta\|=\eta(x),
	\]
	is a normed space. We identify its dual isometrically with the	subspace $
	X_\eta^*=\{\phi\in X^*:\phi|_{N_\eta}=0\} $
	of $X^*$. The corresponding
	dual norm is given by $
	\|\phi\|_\eta
	=
	\sup\{|\phi(x)|:\eta(x)\leq1\}. $
	A subset $\Lambda$ of $B_{X_\eta^*}$ is called \emph{one-norming} for
	$(X,\eta)$ if
	\[
	\eta(x)=\sup_{\phi\in\Lambda}|\phi(x)|,
	\qquad x\in X.
	\]
	For $u\in X$, let $
	J_\eta(u)
	=
	\{\phi\in S_{X_\eta^*}:\phi(u)= \eta(u)\}$
	denotes the set of support functionals at $u$. The abstract numerical
	range of $z\in X$ with respect to $u$ is
	\[
	V_\eta (X,u,z)
	=
	\operatorname{conv}\{\phi(z):\phi\in J_\eta(u)\}.
	\]

	The following observation allows us to transfer the relevant notions
	to the normed quotient.
	
\begin{prop}\label{prop:quotient}
	Let $(X,\eta)$ be a semi-normed space and let $N_\eta = \ker \eta$.  Suppose that  $q : X \to X/N_\eta$ denote the canonical quotient map $q(x) = \widehat{x} = x + N_\eta$. Under the canonical isometric identification $(X/N_\eta)^* \cong X_\eta^*$, support functionals, abstract numerical ranges, and abstract numerical radii are preserved. More precisely, if $u \in X$ with $\eta(u) = 1$ and $z \in X$, then $\phi \in J_\eta(u)$ if and only if $\phi = f \circ q$ for a unique $f \in J(\widehat{u})$. Consequently,
	\[
	V_\eta(X, u, z) = V(X/N_\eta, \widehat{u}, \widehat{z}) \qquad \text{and} \qquad v_u(z) = v_{\widehat{u}}(\widehat{z}),
	\]
	where $\widehat{u} = q(u)$ and $\widehat{z} = q(z)$.
	\end{prop}
		
		\begin{proof}
			The canonical map $q^* : (X/N_\eta)^* \to X_\eta^*$ given by $q^*(f) = f \circ q$ is an isometric linear bijection. For $f \in (X/N_\eta)^*$, we have $\|f\| = \|f \circ q\|_\eta$. Furthermore, for $u \in X$ with $\eta(u) = 1$, $(f \circ q)(u) = f(\widehat{u})$. Hence, $f \in J(\widehat{u})$ if and only if $f \circ q \in J_\eta(u)$. The evaluation $(f \circ q)(z) = f(\widehat{z})$ immediately yields the equalities $V_\eta(X, u, z) = V(X/N_\eta, \widehat{u}, \widehat{z})$ and $v_u(z) = v_{\widehat{u}}(\widehat{z})$.
		\end{proof}
		
		We recall the following characterization of Birkhoff-James 	orthogonality on normed spaces in terms of the abstract numerical range.
		
		\begin{prop}\cite[Th. 2.1]{J}
			\label{prop:James}
			Let $X$ be a normed space and let $x,y\in X$. Then
			\[
			x\perp_B y
			\quad\Longleftrightarrow\quad
			0\in V(X,x,y)
			\quad\Longleftrightarrow\quad
			0\in
			\operatorname{conv}\{x^*(y):x^*\in J(x)\}.
			\]
		\end{prop}
	
	Since for any $x \in X, \eta(x)= \|x+ N_\eta\|,$ it is immediate that $$ u \perp_\eta z \iff u + N_\eta \perp_B z + N_\eta.$$
	Therefore, from Proposition \ref{prop:quotient} and \ref{prop:James}, we get the following. 
	
	\begin{prop}\label{V}
		Let $(X, \eta)$ be a semi normed space and let $u, z \in X.$ Then 
		\[
		u \perp_\eta z \iff 0 \in V_\eta(X, u, z) \iff 0 \in 
		\operatorname{conv}\{\phi(z):\phi\in J_\eta(u)\}.
		\]
	\end{prop}
	
	We now extend the characterization of abstract numerical ranges obtained in \cite[Theorem~2.4]{MMQRS}  for normed spaces to
	semi-normed spaces.
	
	\begin{theorem}\label{thm:semi-normed-numerical-range}
		Let $(X,\eta)$ be a semi-normed space and  let $u\in X$. Suppose that  $\Lambda \subset B_{X^*} \cap N_\eta^\perp$ is one-norming for $X$.
		Then, for every $z\in X$,
		\[
		V_\eta (X,u,z)
		=
		\operatorname{conv}
		\{
		\lim \overline{\phi_n(u)} \phi_n(z):
		\phi_n\in \Lambda, \forall n \,
		\lim |\phi_n(u)|= \eta(u)
		\}.
		\]
	\end{theorem}
	
	\begin{proof}
		Let $N_\eta=\ker\eta$, and write
		$\widehat{u}=q(u)$ and $\widehat{z}=q(z)$. By
		Proposition~\ref{prop:quotient}, the quotient identification preserves
		one-norming subsets and
		\[
		V_\eta (X,u,z)
		=
		V(X/N_\eta,\widehat{u},\widehat{z}).
		\]
		The asserted characterizations therefore follow by applying
		\cite[Th. ~2.4]{MMQRS} to the normed space $X/N_\eta$. 
	\end{proof}

\section{Main results}

\section*{Symmetric points with respect to $v_x(\cdot)$}

\begin{prop}\label{prop:1}
	Let $X$ be a Banach space and let $u\in S_X$. For any $x, y \in X,$
	\[
	x\perp_{v_u}y
	\quad\Longleftrightarrow\quad
	0\in
	\operatorname{conv}
	\{
	\overline{x^*(x)}\,x^*(y):
	x^*\in J(u),\,|x^*(x)|=v_u(x)\}.
	\]
\end{prop}

\begin{proof}
	Recall that $
	J(u)=\{x^*\in S_{X^*}:x^*(u)=1\} $
	and
	\[
	v_u(x)=\sup_{x^*\in J(u)}|x^*(x)|.
	\]
	Thus, $J(u)\subset B_{X^*}$ is a one-norming set for the seminormed
	space $(X,v_u)$. Following Theorem~\ref{thm:semi-normed-numerical-range},
	we obtain
	\[
	V_{v_u}(X,x,y)
	=
	\operatorname{conv}
	\{
	\lim \overline{x_n^*(x)}\,x_n^*(y):
	x_n^*\in J(u),\,
	\lim |x_n^*(x)|=v_u(x)
	\}.
	\]
The set $J(u)$ is weak$^*$-compact, and the map $
x^* \to \big(x^*(x),x^*(y)\big) $
is weak$^*$-continuous. Hence the sets
\begin{eqnarray*}
&	\{
	\lim \overline{x_n^*(x)}\,x_n^*(y):
	x_n^*\in J(u),\
	\lim |x_n^*(x)|=v_u(x)
	\}\\,
 &\quad	\quad \quad \quad =
	\{
	\overline{x^*(x)}\,x^*(y):
	x^*\in J(u),\
	|x^*(x)|=v_u(x)
	\}.
\end{eqnarray*}
Indeed, let $\{x_n^*\} \subset J(u)$ satisfy $
\lim |x_n^*(x)|=v_u(x). $
By weak$^*$-compactness of $J(u)$, $\{x_n^*\}$ has a weak$^*$-convergent
subnet, say $\{x_{n_\gamma}^*\}$, with limit $x^*\in J(u)$. By the weak$^*$-continuity of
$x^* \to x^*(x)$ and $x^* \to x^*(y)$, we have 
\[
|x^*(x)|
=
\lim_\gamma |x_{n_\gamma}^*(x)|
=
v_u(x),
\]
and
\[
\lim_\gamma
\overline{x_{n_\gamma}^*(x)}\,x_{n_\gamma}^*(y)
=
\overline{x^*(x)}\,x^*(y).
\]
Thus, every element of the set on the left-hand side belongs to the set
on the right-hand side. The converse inclusion is immediate, and hence
the two sets are equal.
Consequently,
\[
V_{v_u}(X,x,y)
=
\operatorname{conv}
\{
\overline{x^*(x)}\,x^*(y):
x^*\in J(u),\
|x^*(x)|=v_u(x)
\}.
\]
	Therefore, by Proposition \ref{V} we have 
	\[
	x\perp_{v_u}y
	\quad\Longleftrightarrow\quad
	0\in
	\operatorname{conv}
	\{
	\overline{x^*(x)}\,x^*(y):
	x^*\in J(u),\
	|x^*(x)|=v_u(x)
	\}.
	\]
\end{proof}

	The $v_u$-Birkhoff-James orthogonality is homogeneous. More precisely,
	for every $\lambda\in\mathbb{K}$,  $	 x \perp_{v_u} y	\quad\Longrightarrow\quad	\lambda  x \perp_{v_u} y .$
	Moreover, for every $x, y \in X$, there exists $\alpha\in\mathbb{K}$ such that $\alpha x+ y \perp_{v_u} x.$
Indeed, since $v_u$ is a continuous, the proof of
\cite[Th.~2.1.13]{MPS} applies in the present setting with the norm
replaced by the semi-norm $v_u$.
Similar to  BJ orthogonality in normed space settings, this semi-norm orthogonality is also not symmetric in nature in general. 
Let $ u \in X.$ An element $ x \in X$ is called a \emph{left symmetric point} with
respect to $v_{u}(\cdot)$ if $
x  \perp_{v_u}  y$ implies 
$ y \perp_{v_u}  x ,$ for any  $ y \in X.$
Similarly, $ x $ is called a \emph{right symmetric point} with
respect to $v_{u}(\cdot)$  if $  y  \perp_{v_u} x$ implies  $x \perp_{v_u} y $ for any $  y \in X. $
A point is called \emph{symmetric} if it is both left symmetric and
right symmetric.

\begin{prop}\label{prop:symmetry-smooth-vx}
	Let $X$ be a Banach space and let $ u \in S_X$ be a smooth point with
	$J( u)=\{ u^*\}$. Then every $u\in X$ is  symmetric  with respect to $v_u(\cdot)$.
\end{prop}

\begin{proof}

	Since $u$ is smooth, $v_u(z)=|u^*(z)|$ for every $z\in X$. Thus, for
	any $x,y\in X$,
	\[
	x\perp_{v_u}y	\iff	|u^*(x)+\lambda u^*(y)|	\geq |u^*(x)|
	\quad \forall \lambda \in \mathbb K.
	\]
	If $u^*(x)\neq0$, then necessarily $u^*(y)=0$, since otherwise taking $
	\lambda=-\frac{u^*(x)}{u^*(y)} $
	gives a contradiction. Hence
	\[
	x\perp_{v_u}y	\iff	u^*(x)=0\,\text{or}\,u^*(y)=0.
	\]
	Since the condition
	\[
	u^*(x)=0\, \text{or}\,,u^*(y)=0
	\]
	is symmetric in $x$ and $y$, we have
	\[
	x\perp_{v_u}y	\iff	y\perp_{v_u}x.
	\]
	Therefore, every $x\in X$ is both left symmetric and right symmetric
	with respect to $v_u(\cdot)$.

\end{proof}

\begin{remark}
	Let $u\in S_X$. 
	If $
	B_{X^*}
	=
	\overline{\operatorname{conv}}^{w^*}
	\{\lambda x^*:x^*\in J(u),\,|\lambda|=1\}, $
then $J(u)$ is one-norming for the Banach space $(X,\|\cdot\|)$,i.e., for any $x \in X,$  $
\|x\|=\sup_{x^*\in J(u)}|x^*(x)|.$ Consequently
	\[
	v_u(u)=\|u\|,\qquad u\in X.
	\]
	Thus, in this case, the $v_u$-Birkhoff-James orthogonality coincides
	with the usual Birkhoff-James orthogonality on $X$. In particular, the
	left and right symmetric points with respect to $v_u(\cdot)$ are exactly
	the left and right symmetric points of $X$ under its original norm.
\end{remark}

\begin{remark}
	The spaces $\ell_1^n$ and $\ell_\infty^n$ provide natural instances of
	the preceding observation. Indeed, for every extreme point $u$ of the
	unit ball of either space,
	\[
	v_u(\cdot)=\|\cdot\|.
	\]
	Hence, the $v_u$-Birkhoff-James orthogonality and the corresponding
	notions of left and right symmetric points coincide with those induced
	by the original norm. Thus, the known characterizations of symmetric
	points in \cite{BRS, GSP17}  in these spaces apply directly.
	
	\begin{itemize}
		\item In $\ell_1^n$, there is no nonzero left symmetric point, whereas
		the nonzero right symmetric points are precisely $\{\lambda e_i: 1 \leq i \leq n, \lambda \in \mathbb{K}\},$i.e., the scalar multiples
		of the canonical basis vectors.
		
		\item In $\ell_\infty^n$, the left symmetric points are precisely $\{\lambda e_i: 1 \leq i \leq n, \lambda \in \mathbb{K}\},$ whereas the nonzero
		right symmetric points are precisely the scalar multiples of the
		extreme points of $B_{\ell_\infty^n}$.
	\end{itemize}
\end{remark}

In the general Banach space setting, identifying symmetric points with respect to this semi-norm orthogonality appears to be a rather difficult problem. Nevertheless, in the real Banach space setting, an abstract characterization can be obtained along the same lines as the corresponding characterization in the normed-space setting. The proof follows essentially the same arguments as those used in \cite[Th. 2.1 and 2.2]{SRBB}, and is therefore omitted.

\begin{theorem}
	Let $X$ be a real Banach space and $x \in X. $ Then $u \in X$ is left symmetric with respect to $v_u(.)$ if and only if the following holds
	\begin{itemize}
			\item[(i)]  $ \operatorname{sgn}( x^*(w))  x^*(u) \geq 0,$ for some  $x^* \in J(x), |x^*(w)|= v_u(w)$  $\implies$ $ \operatorname{sgn}( x^*(u)) x^*(w) \geq 0,$ for some  $x^* \in J(x), |x^*(u)|= v_u(u)$.
		
		\item[(ii)] $  \operatorname{sgn}( x^*(u)) x^*(w) \leq 0,$ for some $x^* \in J(x), |x^*(w)|= v_u(w)$ $ \implies$  $  \operatorname{sgn}( x^*(w)) x^*(u) \leq  0,$ for some  $x^* \in J(x), |x^*(u)|= v_u(u)$.
	\end{itemize}
\end{theorem}

\begin{theorem}\label{v-right}
	Let $X$ be a real Banach space and $x \in X. $ Then $u \in X$ is right symmetric with respect to $v_u(.)$ if and only if the following holds
	\begin{itemize}
		\item[(i)]  $ \operatorname{sgn}( x^*(w)) x^*(w) \geq 0,$ for some  $x^* \in J(x), |x^*(u)|= v_u(u)$  $\implies$  $ \operatorname{sgn}( x^*(u))  x^*(u) \geq 0,$ for some  $x^* \in J(x), x^*(w)= v_u(w)$.
		
		\item[(ii)] $  \operatorname{sgn}( x^*(u))x^*(w) \leq  0,$ for some $x^* \in J(x), |x^*(u)|= v_u(u)$  $\implies$  $   \operatorname{sgn}( x^*(w))x^*(u) \leq  0,$ for some $x^* \in J(x), x^*(w)= v_u(w)$.
	\end{itemize}
\end{theorem}

\section*{Symmetric functions in $C(K,X)$}

This section is devoted to the study of symmetric points in
$C(K,X)$ with respect to the numerical-radius semi-norm. We begin by
recording the corresponding characterization of Birkhoff-James
orthogonality, which will be used throughout the section.
Recall that
\[
\Pi(K,X) = \{(k,x^*):k\in K,\,x^*\in J(k)\},
\]
where $J(k) = \{x^*\in S_{X^*}:x^*(k)=1\},$ and 
\[ 
\|f\|_{w} = \sup \{ |x^*(f(k))|: (k,x^*)\in\Pi(K,X) \},
\quad f\in C(K,X).
\]
For each $(k,x^*)\in\Pi(K,X)$, define the evaluation functional
\[
x^*\otimes\delta_k:C(K,X)\longrightarrow\mathbb K
\]
by $
(x^*\otimes\delta_k)(f) = x^*(f(k)),$ for any $
f\in C(K,X). $
Let $
\Lambda =\{ x^*\otimes\delta_k :(k,x^*)\in\Pi(K,X) \}. $ Since
\[
\|f\|_{w} = \sup_{\psi\in\Lambda} |\psi(f)|,
\quad f\in C(K,X),
\]
it follows that $\Lambda$ is a one-norming subset of
$B_{(C(K,X))_{\|\cdot\|_w}^{\,*}}$.  For every $f\in C(K,X)$, define the \emph{numerical-radius norm attainment set} of $f$ by
\[
N_w(f)
=
\{
k\in K:
|x^*(f(k))|=\|f\|_w
\text{ for some }x^*\in J(k)\}.
\]

\begin{prop}
	Let $f\in C(K,X)$. If $	N_w(f)=\{k_0\}, $ 	then $	\|f\|_w =	v_{k_0}(f(k_0)). $
\end{prop}

\begin{proof}
	Since $k_0\in N_w(f)$, there exists $x^*\in J(k_0)$ such that $	|x^*(f(k_0))|	=\|f\|_w.$ 	Hence $	v_{k_0}(f(k_0))	=	\sup_{y^*\in J(k_0)}	|y^*(f(k_0))|	\ge	\|f\|_w.$
	On the other hand
	\[
	v_{k_0}(f(k_0))	=	\sup_{y^*\in J(k_0)}	|y^*(f(k_0))|\leq	\sup_{(k,y^*)\in\Pi(K,X)}
	|y^*(f(k))|	=	\|f\|_w.
	\]
	Therefore, $
	v_{k_0}(f(k_0))	=\|f\|_w. $
\end{proof}

We may therefore apply Theorem~2.2 to obtain the following
characterization of Birkhoff-James orthogonality.

\begin{theorem}\label{BJ-C(K,X)}
	Let $f, g \in C(K,X)$ where $K$ be a compact subset of $S_X.$
 Then $
	f\perp_w g $
	if and only if
	\[
	0	\in\operatorname{conv}	\{  \overline{x^*(f(k))}x^*(g(k)):	(t, x^*)\in\Pi(K,X),	|x^*(f(k))|	=	\|f\|_{w} \}.
	\]
	
\end{theorem}

\begin{proof}
	Set
	\[
	\Lambda=	\{x^*\otimes\delta_k:(k,x^*)\in\Pi(K,X)\}.
	\]
	As observed above, $\Lambda$ is a one-norming subset of
	$B_{(C(K,X))_{\|\cdot\|_w}^{\,*}}$. By Theorem~\ref{thm:semi-normed-numerical-range},
\begin{eqnarray*}
\begin{aligned}
	&V_w(C(K,X),f,g)\\   &\qquad = \operatorname{conv} \{ 	\lim \overline{x_n^*(f(k_n))}\,x_n^*(g(k_n)): 	(k_n,x_n^*)\in\Pi(K,X),\	\lim |x_n^*(f(k_n))|=\|f\|_w \}.
\end{aligned}
\end{eqnarray*}
Observe that $\Pi(K,X)$ is compact with respect to the product of the
norm topology on $K$ and the weak$^*$-topology on $B_{X^*}$, and the
map	 $	(k,x^*)\longmapsto 	\big(x^*(f(k)),x^*(g(k))\big)$ 	is continuous. 	Hence the limiting set above is precisely
	\begin{eqnarray*}
		\begin{aligned}
			&V_w(C(K,X),f,g)\\   &\qquad = \operatorname{conv} \{ 	\lim \overline{x^*(f(k))}\,x^*(g(k)): 	(k,x^*)\in\Pi(K,X),\	\lim |x^*(f(k))|=\|f\|_w \}.
		\end{aligned}
	\end{eqnarray*}
	The result now follows from Proposition~\ref{V}.
\end{proof}

We need the following classical result from point-set topology which will be used repeatedly in this section.

	\begin{lemma}\cite[Cor. 4.3]{Dugundji}\label{Urysohn}
	Let $K$ be a perfectly normal space. Let $A$ and $B$ be disjoint closed subsets of $K$. Then there exists a continuous map $\varphi : K \to [0, 1] $ such that $\varphi^{-1}(0)= A$ and $\varphi^{-1}(1)= B.$
\end{lemma}

We now present our first main result, giving a characterization of
the $nr$-left symmetric functions in $C(K,X)$.

\begin{theorem}\label{left}
	Let $X$ be a Banach space and $K$ be a compact subset of $S_X.$  Let $f \in C(K, X).$ Then $f$ is $nr$-left symmetric if and only if the following holds: 
	\begin{itemize}
		\item[(i)] there exists $k_0 \in K$ such that for any $k \neq k_0$ and for any $x^* \in J(k)$, $x^*(f(k))=0.$
		\item[(ii)] $f(k_0)$ is left symmetric with respect to $v_{k_0}(\cdot).$
	\end{itemize}
\end{theorem}

\begin{proof}
\textbf{Necessary part:}	Let us first prove the necessary part of the theorem. 
	
	(i)
	Since $K$ is compact and $f\neq0$, there exist $k_0 \in K$ and
	$x_0^*\in J(k_0)$ such that
	\[
	|x_0^*(f(k_0))|=\|f\|_w.
	\]
	Suppose, to the contrary, that there exist $k_1\neq k_0$ and
	$x_1^*\in J(k_1)$ such that $
	x_1^*(f(k_1))\neq0. $
	Let
	\[
	|x_1^*(f(k_1))|=r>0.
	\]
	In particular, $v_{k_1}(f(k_1))>0$.
	 Consider the open
	set
	\[
	V=\{k\in K:\|f(k)-f(k_1)\|<\frac r2 \}.
	\]
	Then $k_1\in V$. Since $k_0\neq k_1$ and $K$ is Hausdorff, we may choose an
	open set $U$ such that $
	k_1\in U\subset V $ and $
	k_0\notin U.$ Since $K\subset X$, the space $K$ is perfectly normal and 
	by Lemma~\ref{Urysohn}, there exists
	$\varphi:K\to[0,1]$ such that
	\[
	\varphi^{-1}(1)=\{k_1\},
	\qquad
	\varphi^{-1}(0)=K\setminus U.
	\]
	Define $g\in C(K,X)$ by
	\[
	g(k)=\varphi(k)f(k_1),\qquad k\in K.
	\]
	Since $	|x_1^*(g(k_1))|	=	|x_1^*(f(k_1))|	=	r, $
	we have $\|g\|_w\geq r>0$.
	
	\noindent Let $(k,x^*)\in\Pi(K,X)$ be such that $
	|x^*(g(k))|=\|g\|_w. $
	Since $g(k)=0,$ for any $ k \in K\setminus U$, we have $k\in U$. Moreover,
	\[
	\varphi(k)|x^*(f(k_1))|	=	|x^*(g(k))|	=	\|g\|_w.
	\]
	As $0<\varphi(k)\le1$, we have
	\[
	|x^*(f(k_1))|	=	\frac{\|g\|_w}{\varphi(k)}	\geq	\|g\|_w	\geq r.
	\]
	Since $k\in U\subset V$, we also have
	\[
	|x^*(f(k))-x^*(f(k_1))|	\leq	\|f(k)-f(k_1)\|	<	\frac r2.
	\]
	Therefore,
	\[
	\begin{aligned}
		\operatorname{Re} [
		\overline{x^*(g(k))}x^*(f(k))
		]	&\quad = 	\varphi(k)\operatorname{Re} [
		\overline{x^*(f(k_1))}x^*(f(k))]\\
		& \quad  = 	\varphi(k)\operatorname{Re}\bigg[	|x^*(f(k_1))|^2	+	\overline{x^*(f(k_1))}
		\bigg(x^*(f(k))-x^*(f(k_1))\bigg)
		\bigg]\\
		&\quad\geq 	\varphi(k) \bigg[	|x^*(f(k_1))|^2	-	|x^*(f(k_1))|	|x^*(f(k))-x^*(f(k_1)) |\bigg]\\
		&\quad\geq	\varphi(k)|x^*(f(k_1))|	\bigg[	|x^*(f(k_1))|	-	\|f(k)-f(k_1)\|	\bigg]\\
		&\quad>	\varphi(k)|x^*(f(k_1))|	\frac r2\\
		&\quad >0.
	\end{aligned}
	\]
	Thus, $
	\operatorname{Re} [
	\overline{x^*(g(k))}x^*(f(k)) ]>0 ,$
	for every $(k,x^*)\in\Pi(K,X)$ satisfying $
	|x^*(g(k))|=\|g\|_w. $
	Therefore,
	\[
	0\notin
	\operatorname{conv}
	\{
	\overline{x^*(g(k))}x^*(f(k)):
	(k,x^*)\in\Pi(K,X),\
	|x^*(g(k))|=\|g\|_w \}.
	\]
	Following Theorem~\ref{BJ-C(K,X)}, we get $
	g\not\perp_w f. $
	
	\noindent On the other hand $k_0\notin U$ implies $g(k_0)=0$ and $
	|x_0^*(f(k_0))|=\|f\|_w. $
	Hence
	\[
	\overline{x_0^*(f(k_0))}x_0^*(g(k_0))=0.
	\]
	Consequently,
	\[
	0\in	\{	\overline{x^*(f(k))}x^*(g(k)):	(k,x^*)\in\Pi(K,X),\,	|x^*(f(k))|=\|f\|_w\}.
	\]
	It follows from Theorem~\ref{BJ-C(K,X)} that $	f\perp_w g. $
	This contradicts  that $f$ is $nr$-left symmetric. Consequently,
	$f$ satisfies condition (i).\\

(ii) Suppose, on the contrary, that $f(k_0)$ is not left symmetric with
		respect to $v_{k_0}(\cdot)$. Then there exists $z\in S_X$ such that
		\[
		f(k_0)\perp_{v_{k_0}} z
		\qquad\text{but}\qquad
		z\not\perp_{v_{k_0}} f(k_0).
		\]
			Let $U$ be an open neighbourhood of $k_0$. By Lemma \ref{Urysohn},  there
		exists a continuous function $
		\varphi:K \to [0,1] $
		such that
		\[
		\varphi(k_0)=1,
		\qquad
		\varphi(k)=0,
		\quad
		k\in K\setminus U.
		\]
			Define $g\in C(K,X)$ by $
		g(k)=\varphi(k)z,
		\qquad
		k\in K.$
		
	\noindent 	Since $f(k_0)\perp_{v_{k_0}}z$, from  Proposition~\ref{prop:1} we get
		\[
		0\in
		\operatorname{conv}
		\{
		\overline{x^*(f(k_0))}x^*(z):
		x^*\in J(k_0),\,
		|x^*(f(k_0))|=v_{k_0}(f(k_0))
		\}.
		\]
		As $f$ satisfies condition (i), we have $
		\|f\|_w=v_{k_0}(f(k_0)), $
		and $g(k_0)=z$. Hence
		\[
		0\in
		\operatorname{conv}
		\{
	\overline{x^*(f(k_0))}	x^*(g(k_0)):
		x^*\in J(k_0),\,
		|x^*(f(k_0))|=\|f\|_w
		\}.
		\]
		Therefore, by Theorem~\ref{BJ-C(K,X)}, $
		f\perp_w g. $

	\noindent We now show that $g \not\perp_w f.$	We  first assume that there exists $k_1\neq k_0$ and
		$x^*\in J(k_1)$ such that $
		|x^*(g(k_1))|=\|g\|_w. $
		Since $f$ satisfies condition (i), $
		x^*(f(k_1))=0 $
		for every $x^*\in J(k_1)$. Consequently,
		\[
		0\in	\operatorname{conv}	\{\overline{x^*(g(k))}	x^*(f(k)):	(k,x^*)\in\Pi(K,X),\,|x^*(g(k))|=\|g\|_w	\},
		\]
		and therefore from Theorem~\ref{BJ-C(K,X)} $	g\perp_w f .$
		
	\noindent 	It remains to consider the case when every norming pair of $g$ is of
		the form $(k_0,x^*)$. Then $
		\|g\|_w=v_{k_0}(g(k_0))=v_{k_0}(z). $
		Since $
		z\not\perp_{v_{k_0}}f(k_0), $
		Proposition~\ref{prop:1} gives
		\[
		0\notin
		\operatorname{conv}
		\{
	\overline{x^*(z)}	x^*(f(k_0)):
		x^*\in J(k_0),\,
		|x^*(z)|=v_{k_0}(z)
		\}.
		\]
		As $g(k_0)=z$ and $N_w(g)= \{ k_0\},$  applying Theorem~\ref{BJ-C(K,X)}, we get that $
		g\not\perp_w f, $
		contradicting the assumption that $f$ is $nr$-left symmetric. This
		completes the necessary part of the proof.\\

\noindent \textbf{Sufficient part:}
We next prove the sufficient part. Let $
f\perp_w g. $ Since $f$ satisfies condition (i), we get  $
\|f\|_w=v_{k_0}(f(k_0)). $
Applying Theorem~\ref{BJ-C(K,X)} we get
\[
0\in \operatorname{conv} \{
\overline{x^*(f(k_0))} x^*(g(k_0)):x^*\in J(k_0),\, |x^*(f(k_0))|=v_{k_0}(f(k_0))\}.
\]
By Proposition~\ref{prop:1}, $
f(k_0)\perp_{v_{k_0}}g(k_0). $
Since $f(k_0)$ is left symmetric with respect to $v_{k_0}(\cdot)$, we have $
g(k_0)\perp_{v_{k_0}}f(k_0). $

\noindent Suppose  that there exists $k_1\neq k_0$ and
$x^*\in J(k_1)$ such that $
|x^*(g(k_1))|=\|g\|_w. $
Then, by condition (i), $
x^*(f(k_1))=0, $
and therefore  applying  Theorem~\ref{BJ-C(K,X)}, we easily get $
g\perp_w f $.

\noindent Assume now that every norming pair of $g$ is of the form $(k_0, x^*) \in \Pi(K, X)$. Then $
\|g\|_w=v_{k_0}(g(k_0)). $
Since $
g(k_0)\perp_{v_{k_0}}f(k_0), $
Proposition~\ref{prop:1} gives
\[
0\in \operatorname{conv} \{ \overline{x^*(g(k_0))}x^*(f(k_0)): x^*\in J(k_0),\, 
|x^*(g(k_0))|=v_{k_0}(g(k_0))
\}.
\]
Applying Theorem~\ref{BJ-C(K,X)}, we conclude that $ g\perp_w f. $  Hence $f$ is $nr$-left symmetric.
\end{proof}

Note that whenever each $k \in K$ is a vertex of $B_X,$ both the semi-norm $v_k(.)$ and $\|\cdot\|_w$ actually become the norm. We now present the symmetric points characterizations. 

\begin{cor}
	Let $X$ be a Banach space and $K$ be a compact set of vertices of $B_{X}.$   Let $f \in C(K, X).$ Then $f$ is $nr$-left symmetric if and only if  there exists $k_0 \in K$ such that  
	\begin{itemize}
		\item[(i)]   for any $k \neq k_0$, $f(k)=0.$
		\item[(ii)] $f(k_0)$ is left symmetric with respect to $v_{k_0}(\cdot).$
	\end{itemize}
\end{cor}

We next characterize the $nr$-right symmetric elements in $C(K, X).$  To this end, we first establish the following lemma.

\begin{lemma}\label{lemma1}
	Let $X$ be a Banach space and let  $K$ a compact subset of $S_X$. Suppose that
	$f\in C(K,X)$ and $r \in \mathbb{R}$. For each $k\in K$, define
	\[
	\mathcal A(k)
	=
	\{x^*\in J(k): |x^*(f(k))|= r\}.
	\]
	Then the set-valued map $	\mathcal A(\cdot) :K\rightarrow B_{X^*} $
	is upper semi-continuous with respect to the weak$^*$ topology at
	every $k \in K$.
\end{lemma}

\begin{proof}
	Fix $k_0\in K$. Let $W$ be a weak$^*$-open subset of $B_{X^*}$
	such that $
	\mathcal A(k_0)\subset W. $
	We show that there exists an open neighbourhood $U$ of $k_0$ in $K$
	such that 
	\[
	\mathcal A(k)\subset W
	\qquad\text{for every }k\in U.
	\]
		Suppose on the contrary that no such neighbourhood exists. Then
	there exist a net $\{k_\gamma\} \subset K$ and a net
	$\{x_\gamma^*\} \subset B_{X^*}$ such that $
	k_\gamma\longrightarrow k_0,$
and $	x_\gamma^*\in\mathcal A(k_\gamma)\setminus W, $ for every $\gamma.$
Hence
	\begin{eqnarray}\label{lemma:eqn}
		x_\gamma^*(k_\gamma)=1
	\quad\text{and}\quad
	|	x_\gamma^*(f(k_\gamma))|= r.
	\end{eqnarray}
		Since $B_{X^*}$ is weak$^*$-compact, by passing to a subnet, we
	may assume that $
	x_\gamma^*\overset{w^*}{\longrightarrow}x^* $
	for some $x^*\in B_{X^*}$.
		Since $k_\gamma\to k_0$ and $x_\gamma^*(k_\gamma)=1$, we have
	\[
	\begin{aligned}
		|x^*(k_0)-1|
		&\leq |x^*(k_0-k_\gamma)|
		+|(x^*-x_\gamma^*)(k_\gamma)|  \\
		&\leq \|k_0-k_\gamma\|
		+|(x^*-x_\gamma^*)(k_\gamma)|
		\longrightarrow 0.
	\end{aligned}
	\]
	Thus $
	x^*(k_0)=1, $
	and consequently $x^*\in J(k_0)$.
	Furthermore, since $f(k_\gamma)\to f(k_0)$ and from (\ref{lemma:eqn}) we  get 
	\[
	\begin{aligned}
	\bigg|	|x^*(f(k_0))| - r \bigg|
		&= \bigg| |x^*(f(k_0))| - |x_\gamma^*(f(k_\gamma))\bigg| \\
		&\leq \bigg||x^*(f(k_0))|  - |x^*(f(k_\gamma))| \bigg| + \bigg|  |x^*(f(k_\gamma))| - |x_\gamma^*(f(k_\gamma))|     \bigg|\\
		& \leq 
		|x^*(f(k_0)-f(k_\gamma))| 
		+
		|(x^*-x_\gamma^*)(f(k_\gamma))|
		\longrightarrow0.
	\end{aligned}
	\]
	This implies  $
	|x^*(f(k_0))|= r, $ and therefore
	\[
	x^*\in\mathcal A(k_0)\subset W.
	\]
	Since $W$ is weak$^*$-open and
	$x_\gamma^*\overset{w^*}{\longrightarrow}x^*\in W$, we must have
	$x_\gamma^*\in W$ eventually, which contradicts
	$x_\gamma^*\notin W$ for every $\gamma$.
		Hence there exists an open neighbourhood $U$ of $k_0$ such that
	\[
	\mathcal A(k)\subset W
	\qquad\text{for every }k\in U.
	\]
	Thus $\mathcal A(\cdot)$ is upper semi-continuous at $k_0$.
\end{proof}

We first  provide a necessary condition for $nr$-right symmetric elements in $C(K,X).$

\begin{theorem}\label{necessary}
	Let $X$ be a Banach space and $K$ be a compact subset of $S_X.$  Let $f \in C(K, X).$ Let $f$ is $nr$-right symmetric. Then the following holds: 
	\begin{itemize}
		\item[(i)] for any $k \in K$ there exists $x^* \in J(k)$ such that $|x^*(f(k))| = \|f\|_w.$
		\item[(ii)] for any $k\in K,$ $f(k)$ is right symmetric with respect to $v_{k}(\cdot).$
	\end{itemize}
\end{theorem}

\begin{proof}
	(i)
	Suppose on the contrary we assume that there exists $k_1 \in K$ such that $|x^*(f(k_1))|<\|f\|_w,$ for any $x^* \in J(k_1).$ Let $v_{k_1}(f(k_1)) = r < \|f\|_w.$
	
\textbf{Case~1:}	First we consider the case where $r=0.$  	Since $N_w(f)$ is nonempty and compact and $k_1\notin N_w(f)$,
	choose disjoint open sets $U,V\subset K$ such that $
	N_w(f)\subset U $ and 
	$k_1\in V.$
Following Lemma \ref{Urysohn}, there exist continuous functions $\alpha,\beta:K\to[0,1]$ such that
	\begin{eqnarray*}
			\alpha^{-1}(1)=\{k_1\},
		\qquad
		\alpha^{-1}(0)= K\setminus V,\\
			\beta^{-1}(1)= N_w(f),
		\qquad
		\beta^{-1}(0)= K\setminus U.
	\end{eqnarray*}
	We now define $g \in C(K, X)$ such that 
	\[
	g(k) =  \alpha (k) \|f\|_w k_1 + \beta(k) f(k), \quad \quad \text{for any}\, \, k \in K.
	\]
	Clearly, for any $k \notin U \cup V,$ $g(k) =0.$  Let $k \in U.$ Whenever, $k \in N_w(f), g(k)= f(k),$ therefore, $v_k(g(k))= v_k(f(k)) =\|f\|_w.$ Whenever $k \in U \setminus N_w(f), $ $$v_k(g(k)) = \beta(k) v_k(f(k))	< v_k(f(k))< \|f\|_w.$$ Suppose that $k \in V.$ At $k = k_1, $ $v_{k_1}(g(k_1))= \|f\|_w,$ and whenever $k \in V \setminus \{k_1\},$ we have $v_k(g(k)= \alpha(k)\|f\|_w v_k(k_1) < \|f\|_w.$ Therefore, 
	\[
	\|g\|_w= \|f\|_w  \quad \text{and} \quad  N_w(g)= N_w(f) \cup \{ k_1\}.
	\]
	
	\noindent Observe that at $k \in N_w(f),$ $g(k)= f(k).$ Therefore
	\[
	\big\{ \overline{x^*(f(k))} x^*(g(k)): (k, x^*) \in \Pi(K, X), |x^*(f(k))| = \|f\|_w\big\}= \{ |x^*(f(k))|^2\}= \{ \|f\|_w^2\}.
	\]
	Applying Theorem \ref{BJ-C(K,X)}, we conclude $ f \not\perp_w g.$ 
	On the other hand $ k_1 \in N_w(g),$ and since $v_{k_1}(f(k_1))=0,$ we have 
	\[
	 \{ \overline{x^*(g(k_1))} x^*(f(k_1)) : x^* \in J(k_1),  |x^*(g(k_1))|= \|g\|_w\}=\{0\}.
	\]Therefore,
	\[
	0 \in \operatorname{conv} \{ \overline{x^*(g(k))} x^*(f(k)): (k, x^*) \in \Pi(K, X), |x^*(g(k))|= \|g\|_w \}.
	\]
	Following Theorem \ref{BJ-C(K,X)}, we get that $g \perp_w f.$ This contradicts that $f$ ia not $nr$-right symmetric. 
	
	\textbf{Case~2:} 
	Now suppose that $r>0$. Since the mapping $ k \to  v_k(f(k_1)) $ is continuous and $v_{k_1}(f(k_1))=r>0, $ there exists $\delta>0$ such that $ V:=B_X(k_1,\delta)\cap K, $ satisfying 
\[
v_k(f(k_1))\geq \frac{r}{2} \qquad\text{for every } k\in V.
\]
Moreover, since $k_1\notin N_w(f)$ and $N_w(f)$ is compact, we can take an open set $U\subset K$ containing $N_w(f)$ and satisfying $ U\cap V=\emptyset.$
	Define $\alpha:K\to[0,\infty)$ by 	\[
	\alpha(k)=	\begin{cases}
	\displaystyle	\frac{r}{v_k(f(k_1))}	\bigg(1-\frac{\|k-k_1\|}{\delta} \bigg),
		& k\in V,\\[2ex]	0,	& k\notin V.	\end{cases}
	\]
	Since the mapping $k \to  v_k(f(k_1))$ is continuous and	$v_k(f(k_1))\geq r/2$ on $V$, the function $\alpha$ is continuous on $V$. 	Furthermore, $\alpha(k)\to0$ as $k\to\partial V$. Hence,	$\alpha\colon K\to[0,\infty)$ is continuous. 	Clearly, $	\alpha(k_1)=1, $	and for every $k\in V$,
	\begin{eqnarray}\label{eqn1}
			\alpha(k)v_k(f(k_1))	=	r (1-\frac{\|k-k_1\|}{\delta} ) 	\leq r,
	\end{eqnarray}
	with equality if and only if $k=k_1$.
	
\noindent	By Lemma \ref{Urysohn},  there exists a continuous function 	$\beta:K\to[0,1]$ such that	\[
	\beta^{-1}(1)= N_w(f), 	\qquad 	\beta^{-1}(0)=  K\setminus U.
	\]
	Define $g\in C(K,X)$ by
	\[
	g(k)	=	-\frac{\|f\|_w}{r}\alpha(k)f(k_1)	+\beta(k)f(k),	\qquad k\in K.
	\]
		We first show that $\|g\|_w=\|f\|_w$. If $k\in V$, then	$\beta(k)=0$, and hence $$	g(k)=-\frac{\|f\|_w}{r}\alpha(k)f(k_1). $$	Therefore following (\ref{eqn1})
	\[
	v_k(g(k))	=	\frac{\|f\|_w}{r}\alpha(k)v_k(f(k_1))	=	\|f\|_w (1-\frac{\|k-k_1\|}{\delta} )	\leq \|f\|_w.
	\]
	In particular, $ 	g(k_1)=-\frac{\|f\|_w}{r}f(k_1) $	and $$
	v_{k_1}(g(k_1))	=	\frac{\|f\|_w}{r}v_{k_1}(f(k_1))	= \|f\|_w.$$
	Moreover, for any $k \in V \setminus \{k_1\},$ $\alpha(k) v_k(f(k_1))< r,$ we have  $	v_k(g(k))< \|f\|_w.$
	
	\noindent 	If $k\in U$, then $\alpha(k)=0$, and therefore $	g(k)=\beta(k)f(k). $ Consequently, $$
	v_k(g(k))	=	\beta(k)v_k(f(k)) \leq  \|f\|_w.$$
	For $k\in N_w(f)$, we have $\beta(k)=1$, and hence $	v_k(g(k))= \|f\|_w.$
	Finally, if $k\notin U\cup V$, then $g(k)=0$.	Thus 
	\[
		\|g\|_w= \|f\|_w \quad \text{and} \quad 	N_w(g)=N_w(f)\cup\{k_1\}. 
	\]
	Observe that at $k \in N_w(f),$ $g(k)= f(k).$ Therefore
\[
\big\{ \overline{x^*(f(k))} x^*(g(k)): (k, x^*) \in \Pi(K, X), |x^*(f(k))| = \|f\|_w\big\}= \{ |x^*(f(k))|^2\}= \{ \|f\|_w^2\}.
\]
Applying Theorem \ref{BJ-C(K,X)}, we conclude $ f \not\perp_w g.$

\noindent On the other hand take $k' \in N_w(f) \subset N_w(g),$ we have $g(k')= f(k'). $ Then for any $x^* \in J(k')$, 
\[
\{ \overline{x^*(g(k'))} x^*(f(k')): x^* \in J(k'), |x^*(g(k'))|= \|g\|_w\} = \{ |x^*(f(k'))|^2\}= \{ \|f\|^2_w\}.
\]
As $k_1 \in N_w(f),$ let $x^*\in J(k_1)$ satisfy $
	|x^*(g(k_1))|=\|g\|_w=\|f\|_w.$
	Since
	\[
	g(k_1)=-\frac{\|f\|_w}{r}f(k_1),
	\]
	we obtain
	\begin{eqnarray*}
			\overline{x^*(g(k_1))}x^*(f(k_1)) = - \frac{\|f\|_w}{r} |x^*(f(k_1)) | < 0
				\end{eqnarray*}
	Combining these two observation, 	we get that 
	\[
	0\in\operatorname{conv}\{ \overline{x^*(g(k))} x^*(f(k)): (k, x^*) \in \Pi(K, X), |x^*(g(k))|= \|g\|_w\}.
	\]
	Therefore, by Theorem~\ref{BJ-C(K,X)}, $
	g\perp_w f. $
		This contradicts the right symmetry of $f$.\\

	(ii) Suppose on the contrary we assume that there exists $k_0 \in K$ such that $f(k_0)$ is not right symmetric with respect to $v_{k_0}(\cdot).$ Then there exists $w \in X$  and 
	\[
	w \perp_{v_{k_0}} f(k_0)	\qquad\text{but}\qquad	f(k_0)\not\perp_{v_{k_0}} w.
	\]
	Since the semi-norm orthogonality is homogenous, we can choose $w$ such that $v_k(w) \leq \|f\|_w.$    Consider the set 
	\[
	\mathcal{A}(k_0)=\{ x^* \in J(k_0): |x^*(f(k_0))|= \|f\|_w\}. 
	\]
	As $	f(k_0)\not\perp_{v_{k_0}} w$, from Proposition \ref{prop:1} we get 
	\[
	0 \notin \operatorname{conv} \{  \overline{x^*(f(k_0))}x^*(w): x^* \in J(k_0), |x^*(f(k_0))|= \|f\|_w\}.
	\]
	This implies that there exists $ \theta \in \mathbb{R} $ such that $
	\operatorname{Re} \bigg(e^{i \theta}\overline{x^*(f(k_0))}x^*(w) \bigg)> 0, \forall x^* \in  \mathcal{A}(k_0).$ Take $z =e^{i\theta} w,$ clearly $v_k(z)= v_k(w) \leq \|f\|_w.$ Then
	\[
	\operatorname{Re}	(	\overline{x^*(f(k_0))}x^*(z) )>0	\qquad	\text{for every }x^*\in\mathcal A(k_0).
	\]
	Define $
	\Phi:B_{X^*}\longrightarrow\mathbb R $
	by
	\[
	\Phi(x^*)	=	\operatorname{Re}( 	\overline{x^*(f(k_0))}x^*(z)).
	\]
	Since $\Phi$ is weak$^*$-continuous and $\mathcal A(k_0)$ is weak$^*$-compact,
	we have
	\[
	c=\min_{x^*\in\mathcal A(k_0)}\Phi(x^*)>0.
	\]
	Hence the set
	\[
	W=\{ 	x^*\in B_{X^*}:\Phi(x^*)>\frac{c}{2} \}
	\]
	is a weak$^*$-open neighborhood of $\mathcal A(k_0)$. By Lemma~\ref{lemma1},
	there exists an open neighborhood $U$ of $k_0$ such that $
	\mathcal A(k)\subset W$, for any $k\in U.$ 
	Following Lemma \ref{Urysohn}, there exists $\alpha: K \to [0,1]$ such that $\alpha(k_0)=1$ and $\alpha(K\setminus U)=0.$ Define $g: K \to X$ such that 
	$$g(k)= (1- \alpha(k)) f(k)+ \alpha(k) z, \quad \forall k \in K.$$
	
\noindent	As $f$ satisfies the condition (i), we have $v_k(f(k))=\|f\|_w, $ for any $k \in K.$
Observe that whenever  $k \in K \setminus U, $ $v_{k}(g(k))= v_{k}(f(k))= \|f\|_w.$ Moreover, for any $k \in U,$ $$v_{k}(g(k))= v_{k}[(1- \alpha(k) f(k) +  \alpha(k) z)] \leq (1-\alpha(k)) v_{k}(f(k))+\alpha(k) v_k(z) \leq \|f\|_w. $$
	So clearly, $\|g\|_w =\sup_{k\in K}	 v_{k} (g(k)) = \|f\|_w. $ Also, $v_{k_0}(g(k_0)) =\|g\|_w.$ 
	\begin{eqnarray*}
		\|g + \lambda f \|_w= \sup_{k \in K} v_k(g(k)+ \lambda f(k)) &\geq& v_{k_0}( g(k_0)+ \lambda f(k_0))\\ &=& v_{k_0}(z+ \lambda f(k_0)) \\ &\geq&  v_{k_0} (z)= v_{k_0}(g(k_0))= \|g\|_w.
	\end{eqnarray*}
	This implies $g \perp_w f.$

\noindent On the other hand since	for any $k \in K \setminus U, $ $ g(k)= f(k),$ we get
	\[
	\{ x^*(g(k)): k \in K \setminus U, x^* \in J(k), x^*(f(k))= \|f\|_w\}= \{  \|f\|_w\}.
	\] 
	For any $k \in U, $ $g(k)= (1 -\alpha(k)) f(k) + \alpha(k) z.$ Observe that for any $x^* \in \mathcal{A}(k_0)$ we have  $x^* \in W,$ so, $ \operatorname{Re}\bigg(
	\overline{x^*(f(k))}x^*(z) \bigg)> 0.$
	Consequently,
	\begin{align*}
		\operatorname{Re}	(	\overline{x^*(f(k))}x^*(g(k)))	&=	(1-\alpha(k))|x^*(f(k))|^2+	\alpha(k)	\operatorname{Re}	(	\overline{x^*(f(k))}x^*(z))\\
			&>0.
	\end{align*}
	This implies 
	\[
	0 \notin \operatorname{conv} \{ \overline{x^*(f(k))}x^*(g(k)):(k, x^*) \in \Pi(K, X), |x^*(f(k))|= \|f\|_w\}.
	\]
	Following Theorem \ref{BJ-C(K,X)}, we get $f \not\perp_w g. $
This contradicts that $f$ is $nr$-right symmetric. 
\end{proof}

The converse implication holds when $X$ is a real Banach space, and hence we obtain the following characterization.

\begin{theorem}\label{right}
	Let $X$ be a real Banach space and $K$ be a compact subset of $S_X.$  Let $f \in C(K, X).$  Then $f$ is nr-right symmetric if and only if  the following holds: 
	\begin{itemize}
		\item[(i)] for any $k \in K$ there exists $x^* \in J(k)$ such that $|x^*(f(k))| = \|f\|_w.$
		\item[(ii)] for any $k\in K,$ $f(k)$ is right symmetric with respect to $v_{k}(\cdot).$
	\end{itemize}
\end{theorem}

\begin{proof}
Following Theorem~\ref{necessary}, it suffices to prove the sufficient part.
Let $g\perp_w f$. By Theorem~\ref{BJ-C(K,X)}, we have
\[
0\in
\operatorname{conv}
\{
sgn(x^*(g(k))) x^*(f(k)):
k\in K,\,x^*\in J(k),\
|x^*(g(k))|=\|g\|_w \}.
\]
Hence, there exist $k_1,k_2\in K$ and
$x_1^*\in J(k_1)$, $x_2^*\in J(k_2)$ such that
\[
|x_1^*(g(k_1))|=|x_2^*(g(k_2))|=\|g\|_w
\]
and
\[
sgn(x_1^*(g(k_1)))x_1^*(f(k_1))\geq0,
\qquad
sgn(x_2^*(g(k_2)))x_2^*(f(k_2))\leq0.
\]
By condition {\rm (i)}, we have $
v_k(f(k))=\|f\|_w,$ for any $k\in K. $
Since $f(k_1)$ is right symmetric with respect to
$v_{k_1}(\cdot)$,   following Theorem \ref{v-right}, we get that there exists
$y_1^*\in J(k_1)$ such that
\[
|y_1^*(f(k_1))|
=v_{k_1}(f(k_1))
=\|f\|_w,
\qquad
sgn(y_1^*(f(k_1)))y_1^*(g(k_1))\geq0.
\]
Similarly, since $f(k_2)$ is right symmetric with respect to
$v_{k_2}(\cdot)$, again from Theorem \ref{v-right} we get that there exists  $y_2^*\in J(k_2)$ such that
\[
|y_2^*(f(k_2))|
=v_{k_2}(f(k_2))
=\|f\|_w,
\qquad
sgn(y_2^*(f(k_2)))y_2^*(g(k_2))\leq0.
\]
Consequently,
\[
0\in
\operatorname{conv}
\{
sgn(x^*(f(k)))x^*(g(k)):
k\in K,\,x^*\in J(k),\
|x^*(f(k))|=\|f\|_w \}.
\]
Applying  Theorem~\ref{BJ-C(K,X)} we get  $
f\perp_w g.$
Therefore, $f$ is $nr$-right symmetric..
\end{proof}

We finish this section by recording some straightforward consequences of
the above characterization, which also serve as illustrative examples.

\begin{cor}\label{l_1}
	Let $X = \ell_1^n$ and let $K = \{ e_i: 1 \leq i \leq n\}.$
	Suppose that $f \in C(K, X).$ Then the following holds
	\begin{itemize}
		\item[(i)] $f$ is $nr$-left symmetric if and only if $f=0.$
		\item[(ii)] $f$ is $nr$-right symmetric if and only if
		$f(e_i) = \pm e_j$ for any $1 \leq i,j \leq n.$
	\end{itemize}
\end{cor}

\begin{cor}
	Let $X = \ell_\infty^n$ and let $K = \operatorname{ext} B_X,$ the set of all extreme point of $B_X.$
	Suppose that $f \in C(K, X).$ Then the following holds
	\begin{itemize}
		\item[(i)] $f$ is $nr$-left symmetric if and only if there exists
		$k_0 \in \operatorname{ext} B_X$ such that $f(k_0)=\pm e_j$ for some
		$1\leq j\leq n$ and
		$f(k)=0$ for every $k\in\operatorname{ext} B_X\setminus\{k_0\}$.
		\item[(ii)] $f$ is $nr$-right symmetric if and only if
		$f(k)\in K$ for every $k\in\operatorname{ext} B_X$.
	\end{itemize}
\end{cor}

The Hilbert-space setting admits a particularly simple formulation,
since the support functional at each point of $S_{\mathcal H}$ is unique.

\begin{cor}
	Let $\mathcal{H}$ be a Hilbert space and let $K$ be a compact subset of
	$S_{\mathcal{H}}$. For $f\in C(K,\mathcal{H})$, the following statements
	hold:
	\begin{itemize}
		\item[(i)] $f$ is $nr$-left symmetric if and only if there exists
		$k_0\in K$ such that
		$\langle f(k),k\rangle=0$ for every
		$k\in K\setminus\{k_0\}$.
		
		\item[(ii)] $f$ is $nr$-right symmetric if and only if
		$|\langle f(k),k\rangle|=\|f\|_w$ for every $k\in K$.
	\end{itemize}
In addition,  if $K$ is connected, then every $nr$-left symmetric
	$f\in C(K,\mathcal H)$ satisfies $\|f\|_w=0$.
\end{cor}

\vspace{.2cm}

\section*{Application to space of compact operators}

Let $\operatorname{ext}(B_{X^*})$ be the set of all extreme points of $B_{X^*}.$ We define
$E_{X^*}$ to be a maximal subset of $\operatorname{ext}(B_{X^*})$
containing no two distinct unimodular scalar multiples; equivalently,
\[
E_{X^*}\subseteq\operatorname{ext}(B_{X^*}),\qquad
x^*,y^*\in E_{X^*},\quad
y^*=\lambda x^*,  |\lambda|=1
\quad\Longrightarrow\quad
x^*=y^*,
\]
and
\[
\operatorname{ext}(B_{X^*})
=
\{\lambda x^*:x^*\in E_{X^*},\,|\lambda|=1 \}.
\]

\begin{prop}\label{prop:K(X)}
	Let $X$ be a reflexive Banach space and let $E_{X^*}$ be a set of
	representatives of $\operatorname{ext}(B_{X^*})$ modulo unimodular
	scalar multiples. Then
	\[
	\Phi:\mathcal K(X)\longrightarrow
	C(E_{X^*},X^*),
	\qquad
	\Phi(T)(x^*)=T^*x^*,
	\]
	defines a conjugate-linear isometric embedding when
	$C(E_{X^*},X^*)$ is endowed with the semi-norm $\|\cdot\|_w$.
\end{prop}

\begin{proof}
	Since $X$ is reflexive,  it is easy to observe that 
	 for $T\in\mathcal K(X)$,
	\[
	\begin{aligned}
		\|\Phi(T)\|_w
		&=
		\sup  \{
		|x^*(Tx)|:
		x\in S_X,\,x^*\in E_{X^*},\,x^*(x)=1 \}\\
	&=
		\sup \{
		|x^*(Tx)|:
		x\in S_X,\,x^*\in\operatorname{ext}(B_{X^*}),\
		x^*(x)=1 \}.
	\end{aligned}
	\]
	For each $x\in S_X$, consider the weak$^*$-compact face $
	F_x=\{x^*\in B_{X^*}:x^*(x)=1\}. $
	By the Krein-Milman theorem, $
	F_x=\overline{\operatorname{conv}}^{\,w^*}\operatorname{ext}(F_x),$
	and $
	\operatorname{ext}(F_x)
	\subseteq
	\operatorname{ext}(B_{X^*})\cap F_x. $
	Since $x^*\to x^*(Tx)$ is weak$^*$-continuous, we obtain
	\[
	\sup_{x^*\in F_x}|x^*(Tx)|
	=
	\sup_{\substack{x^*\in\operatorname{ext}(B_{X^*})\\x^*(x)=1}}
	|x^*(Tx)|.
	\]
	Therefore,
	\[
	\begin{aligned}
		\|\Phi(T)\|_w
	&	=	\sup\{	|x^*(Tx)|: x\in S_X, x^*\in B_{X^*}, x^*(x)=1\}\\
	&	=
		w(T).
	\end{aligned}
	\]
		It remains to verify that $\Phi(T)\in C(E_{X^*},X^*)$.  
		Let $\{x_\alpha^*\}$
	be a net in $E_{X^*}$ such that $	x^*_\alpha\overset{w^*}{\longrightarrow}x^*.$
		As $X$ is reflexive and $T^*$ is compact, $
	\|T^*x^*_\alpha-T^*x^*\|\longrightarrow0. $
	Thus $x^*\to T^*x^*$ is norm-continuous on
	$E_{X^*}$. 
		Finally, for $\alpha,\beta\in\mathbb{K}$ and $S,T\in\mathcal{K}(X)$,
	\[
	\begin{aligned}
		\Phi(\alpha T+\beta S)(x^*)
		=
		(\alpha T+\beta S)^*x^*
		=
		\overline{\alpha}\,T^*x^*
		+\overline{\beta}\,S^*x^*,
	\end{aligned}
	\]
	and hence
	\[
	\Phi(\alpha T+\beta S)
	=
	\overline{\alpha}\,\Phi(T)
	+\overline{\beta}\,\Phi(S).
	\]
 Hence $\Phi$
	is a conjugate-linear isometric embedding.
\end{proof}

Since $\Phi$ is a conjugate-linear isometry, it preserves Birkhoff-James
orthogonality and consequently, the local symmetric points. In particular,
if $\Phi(T)$ is left (respectively, right) symmetric, then $T$ is left
(respectively, right) symmetric. Hence, the results obtained in the
previous section for $C(K, X)$ can be applied to $\Phi(T)$ and
transferred back to $T$.

\begin{prop}
	Let $X$ be a reflexive Banach space. Suppose that there exists $x_0^* \in E_{X^*}$ such that
	\begin{itemize}
		\item[(i)] for any $x^* \in E_{X^*} \setminus\{x_0^*\}$ and for any $x \in S_{X}$ with $x^*(x)=1,$ we have  $ x^*(Tx)=0$ and  
		
		\item[(ii)] $T^*(x_0^*)$ is left symmetric with respect to $v_{x_0^*}(\cdot).$
	\end{itemize}
	Then $T \in \mathcal{K}(X)$ is $nr$-left symmetric. 
\end{prop}

\begin{prop}
Let $X$ be a reflexive real Banach space. Suppose that for any $x^* \in E_{X^*} $ the following hold
\begin{itemize}
	\item[(i)]  there exists $x \in S_X$ such that $x^*(x)=1,$ and   $ |x^*(Tx)|=\|T\|_v.$ 
	
	\item[(ii)]  $T^*(x^*)$ is left symmetric with respect to $v_{x^*}(\cdot).$
\end{itemize}
Then $T \in \mathcal{K}(X)$ is $nr$-right symmetric. 
\end{prop}

As a further  application of our results, we recover the characterization
of the left and right symmetric points of
$\mathcal{L}(\ell_1^n)$ and $\mathcal{L}(\ell_\infty^n)$ obtained in
\cite{GMPS}. Indeed, these results follow directly from
Corollary~\ref{l_1}, which gives the corresponding characterization for
$C(E,\ell_1^n)$.\\

\noindent \textbf{Symmetric points of $\mathcal{L}(\ell_\infty^n)$:}
Whenever $X= \ell_\infty^n$,  we can take $E_{X^*}= \{e_1, \ldots, e_n\}.$ Since each of $e_i$ is a vertex, the numerical radius defined on $\mathcal{L}(X)$ is a norm. Observe that  the map defined in Proposition \ref{prop:K(X)} is also surjective, which makes $\mathcal{L}(\ell_\infty^n)$ with the numerical radius norm  is a isometrically isomorphic to $C(E_{X^*}, \ell_1^n),$ endowed with the norm $\|\cdot\|_w.$ Hence, the characterization of left and right symmetric points follows immediately from Corollary \ref{l_1}.

\begin{theorem}
Let 	$T \in \mathcal{L}(\ell_\infty^n).$ Then the following holds: 
\begin{itemize}
\item [(i)]  $T$ is $nr$-left symmetric if and only if $T=0.$ 
\item[(ii)] $T$ is $nr$-right symmetric if and only if $T^*(e_i)= \pm e_j,$ where $i,j \in \{1,2, \ldots, n\}.$
\end{itemize} 
\end{theorem}

\noindent\textbf{Symmetric points of $\mathcal{L}(\ell_1^n)$.}
Let $X=\ell_1^n$ and $E_X=\{e_1,\ldots,e_n\}$. Again, each $e_i$ is a
vertex, and hence the numerical radius is a norm on
$\mathcal{L}(\ell_1^n)$. The map
\[
\Phi:\mathcal{L}(\ell_1^n) \to C(E_X,\ell_1^n),
\qquad
\Phi(T)(e_i)=Te_i,
\]
is a surjective linear isometry with respect to the numerical radius norm
on $\mathcal{L}(\ell_1^n)$ and $\|\cdot\|_w$ on $C(E_X,\ell_1^n)$.
Consequently, Corollary~\ref{l_1} gives the following characterization.

\begin{theorem}
	Let 	$T \in \mathcal{L}(\ell_1^n).$ Then the following holds: 
	\begin{itemize}
		\item [(i)]  $T$ is $nr$-left symmetric if and only if $T=0.$ 
		\item[(ii)] $T$ is $nr$-right symmetric if and only if $T(e_i)= \pm e_j,$ where $i,j \in \{1,2, \ldots, n\}.$
	\end{itemize} 
\end{theorem}

\section*{Conclusion}

The results obtained in this article  provide a structural description of symmetric points arising from numerical-radius orthogonality in $C(K,X)$. In comparison with the symmetric points arising from the supremum norm, studied in \cite{PSS}, and the $\rho$-orthogonal symmetric points considered in \cite{SSP}, the characterizations
obtained here exhibit a distinct structural behaviour. In particular, the numerical-radius structure allows the symmetricity of a function to be described through the supporting functionals associated with the points of $K$. 
It is worth mentioning that the characterization of symmetric operators in $\mathcal K(X)$ and $\mathcal L(X)$ remains open in the most general setting, for both the operator norm and the numerical radius norm.   In the Hilbert-space setting, for the numerical radius norm, the left symmetric operators on both $\mathcal K(\mathcal H)$ and $\mathcal L(\mathcal H)$ have been characterized (see \cite{CS}), while the corresponding characterization of the right symmetric operators is still elusive.  These questions provide natural directions for further study of the
interaction between orthogonality, symmetricity, and the geometry of
operator spaces.

\vspace{0.5cm}
\noindent \textbf{Acknowledgment:}  The author likes to  thank  Prof. Kallol Paul and Dr. Debmalya Sain for their helpful comments during the preparation of this article.

\end{document}